%% file: paper.tex
\documentclass[a4paper]{article}

\usepackage{amsmath,amsthm}
\usepackage{color}
\usepackage{tikz}
\usepackage{url}
\usepackage{booktabs}

\newtheorem{proposition}{Proposition}[section]
\newtheorem{theorem}[proposition]{Theorem}
\newtheorem{lemma}[proposition]{Lemma}
\newtheorem{remark}[proposition]{Remark}
\newtheorem{corollary}[proposition]{Corollary}

\newtheorem{definition}[proposition]{Definition}
\newcommand{\ZZ}{\mathbf{Z}}
\newcommand{\RR}{\mathbf{R}}
\newcommand{\sett}[1]{\{{#1}\}}
\newcommand{\set}[2]{\{{#1}\,|\,{#2}\}}
\newcommand{\norm}[2][]{\|{#2}\|_{#1}}
\newcommand{\abs}[1]{|{#1}|}
\newcommand{\scp}[3][]{\langle {#2},{#3}\rangle_{#1}}
\newcommand{\bdistance}[3][]{D_{#1}({#2},{#3})}
\newcommand{\closure}[1]{\overline{#1}}

\DeclareMathOperator*{\argmin}{argmin}
\DeclareMathOperator{\dom}{dom}
\DeclareMathOperator{\range}{range}
\DeclareMathOperator{\sign}{sign}

\title{Heuristic parameter-choice rules for convex variational
  regularization based on error estimates}

\author{Bangti Jin\thanks{Zentrum f\"ur Technomathematik,
    Universit\"at Bremen, Fachbereich Mathematik/Informatik, Postfach
    33 04 40, D-28334 Bremen, Germany ({\tt
      btjin@math.uni-bremen.de})} \and Dirk A. Lorenz\thanks{Institute for Analysis and Algebra, Pockelsstr. 14,
    D-38118 Braunschweig, Germany ({\tt d.lorenz@tu-braunschweig.de}).
    Dirk A.~Lorenz is supported by the DFG under grant LO 1436/2-1
    within the Priority Program SPP 1324 ``Extraction of quantitative
    information in complex systems''.}}

\begin{document}
\maketitle
\begin{abstract}
  In this paper, we are interested in heuristic parameter choice rules
  for general convex variational regularization which are based on
  error estimates. Two such rules are derived and generalize those
  from quadratic regularization, namely the Hanke-Raus rule and
  quasi-optimality criterion. A posteriori error estimates are shown
  for the Hanke-Raus rule, and convergence for both rules is also
  discussed. Numerical results for both rules are presented to
  illustrate their applicability.
\end{abstract}



\section{Introduction}
We consider the ill-posed problem of determining a solution $x$ to
\begin{equation}\label{eqn:invprob}
Kx = y^\delta,
\end{equation}
when only a noisy version $y^\delta$ of the exact data $y^\dagger$
is available, which furthermore satisfies an inequality
$\|y^\delta-y^\dagger\|\leq\delta$. In our setting $K:X\to Y$ is a
bounded and linear operator mapping from a Banach space $X$ into a
Hilbert space $Y$.

As usual for inverse problems, the numerical solution of problem
\eqref{eqn:invprob} suffers from ill-posedness. In particular, a
small change in the data $y^\delta$ can lead to an enormous
deviation of the solution $x$. To combat the inherent instability,
regularization has been established as an effective approach since
the pioneering work of Tikhonov \cite{Tikhonov1977illposed}. The
regularization method under consideration is general convex Tikhonov
regularization, i.e.,~for a convex and (weak) lower semicontinuous
functional $R:X\to [0,\infty]$, we seek a minimizer, denoted by
$x_\alpha^\delta$, of the functional
\begin{equation}
  \label{eq:J_alpha}
  \mathcal{J}_\alpha(x)=\tfrac{1}{2}\|Kx-y^\delta\|^2+\alpha R(x),
\end{equation}
and takes the minimizer $x_\alpha^\delta$ as an approximate solution
to the unknown exact solution $x^\dagger$. Here $R$ is the
regularization functional incorporating \textit{a priori}
information, and $\alpha$ is known as the regularization parameter,
determining the tradeoff between the data fitting term and the
regularization term.

Tikhonov regularization formulations of this form have attracted
considerable interest in recent years, and have found applications
in diverse disciplines, e.g., imaging science
\cite{rudin1992tv,chan2006imageprocessing} and signal processing
\cite{donoho2006compressedsensing,candes2005decoding}. Because of
their immense practical importance, the functional
$\mathcal{J}_\alpha$ has been the subject of many recent
investigations. Theoretically speaking, since the pioneering work
\cite{burger2004convarreg}, convergence and convergence rates under
a variety of conditions have been established
\cite{resmerita2005regbanspaces,hofmann2007convtikban,lorenz2008reglp,grasmair2008sparseregularization}.
Numerically, several efficient algorithms have also been proposed
\cite{chambolle2004tvprojection,griesse2008ssnsparsity,Wright2009}.

But one of the most important questions in applying these techniques
to practical problems, i.e., choosing an appropriate regularization
parameter $\alpha$, remains largely underexplored. While the problem
of parameter choice has been discussed in depth for the conventional
quadratic regularization, see e.g., \cite{engl1996inverseproblems}
for theoretical studies and \cite{hansen:1998,vogel:2002inverse} for
details on numerical implementation, the case of general convex
regularization has scarcely been addressed. As to existing studies
on parameter selection for Tikhonov regularization in Banach space,
we are aware of Morozov's discrepancy principle \cite{Morozov1966},
which was recently investigated \cite{Bonesky2009,jin2009c}. Some
theoretical results, e.g., convergence and convergence rates, were
derived. In the latter work, an algorithm for solving the
discrepancy equation was also proposed. However, the discrepancy
principle requires an estimate of the noise level, which is not
always available. Therefore, there is a significant interest in
deriving heuristic choice rules which do not require a knowledge of
the exact noise level and still allow some theoretical
justification. One such rule is due to the authors
\cite{ItoJinZou:2008}, where existence of a solution and a
posteriori error estimates are derived. Another is the balancing
principle, recently derived using the model function approach in
\cite{Jin2009d}, for a model with $L^1$ data fitting and quadratic
regularization. 

In the present study, we shall derive two heuristic choice rules
based on error estimates, which are achieved by a refined analysis of
regularization process. Error estimate-based heuristic choice rules
are well-known for the conventional quadratic regularization
\cite{engl1996inverseproblems}, but to the best of the authors'
knowledge, there is no known rule of this type for general convex
variational regularization. The derived rules generalize Hanke-Raus
rule and quasi-optimality criterion for quadratic regularization to
general convex regularization. Some theoretical justifications,
e.g., existence, a posteriori error estimate and convergence, of
both rules are provided. Numerical results are presented to validate
some theoretical findings and to illustrate the features of both
rules.

\paragraph{Notation:} The linear operator $K:X\to Y$ from a Banach
space $X$ into a Hilbert space $Y$ is assumed to be bounded;
$K^*:Y\to X^*$ denotes its adjoint operator. We assume that the exact
data $y^\dagger$ is attainable, i.e., $y^\dagger\in\range K$, and
the noisy data $y^\delta$ satisfies $\norm{y^\dagger-y^\delta}\leq
\delta$. The functional $R:X\to[0,\infty]$ is assumed to be proper,
convex, weakly lower semicontinuous and coercive. This conditions
ensure that the functional $\mathcal{J}_\alpha$ defined
in~\eqref{eq:J_alpha} possess minimizers (cf.~\cite{hofmann2007convtikban}).
We shall denote by $x_\alpha^\delta$ a minimizer to the
functional $\mathcal{J}_\alpha$,
and by $x_\alpha$ a corresponding
minimizer for exact data $y^\dagger$, i.e.,
\[
x_\alpha \in \argmin\left\{ \tfrac12\norm{Kx-y^\dagger}^2 + \alpha
R(x)\right\}.
\]
By $x^\dagger$ we denote a minimum-$R$ solution of the equation
$Kx=y^\dagger$ (see e.g.~\cite{hofmann2007convtikban}).
With $\partial R(x)$ we denote the subdifferential of a
convex functional $R$ at $x$ \cite{ekelandtemam1976convex}.
Throughout the paper we assume that the exact
solution $x^\dagger$ fulfills the following source condition
(see~\cite{burger2004convarreg}):
\begin{equation}
  \label{eq:source_condition1}
  \exists w: K^*w\in\partial R(x^\dagger).
\end{equation}
For any
$\xi\in\partial R(x)$, we denote the Bregman distance from $x$ to
$x'$ with respect to $\xi$ with
\[
\bdistance[\xi]{x'}{x} = R(x') - R(x) - \scp{\xi}{x'-x}.
\]
We note that the Bregman distance $\bdistance[\xi]{x'}{x}$ is always
nonnegative, although in general it can vanish for distinct $x'$ and
$x$. Bregman distance provides a natural measure of various errors,
and for a detailed discussion, we refer to
\cite{BatnariuResmerita2006}.

\section{Estimates for different errors}
\label{sec:estim-diff-errors}

In the case of regularization in Hilbert spaces, one usually splits
the total error, i.e.,~the distance from $x_\alpha^\delta$ to
$x^\dagger$, into the \emph{approximation error} and the
\textit{data error},  which refer to the distance from $x_\alpha$ to
$x^\dagger$ and that from $x_\alpha^\delta$ to $x_\alpha$,
respectively, cf.~\cite{engl1996inverseproblems}. This is achieved
with the help of a triangle inequality. Then the approximation error
and the data error are estimated separately to get an estimate for
the total error. Theoretically, the behavior of the approximation
error contains information about how difficult it is to approximate
the unknown solution $x^\dagger$ and provides hints on what
conditions on $x^\dagger$ may be helpful. The behavior of the data
error shows how noise influences the accuracy of the reconstruction.

In the case of convex variational regularization one usually
estimates the total error directly. One reason is that in this
setting the natural distance measure for the errors is the Bregman
distance which does not fulfill the triangle inequality. In this
section we provide estimates for different terms. This sheds
insights in the regularization process and thereby shows that a
splitting into approximation and data error is still useful.

\begin{proposition}
Let $x^\dagger$ fulfill the source
condition~\eqref{eq:source_condition1} with $\xi=K^*w$. Then the
approximation error and the corresponding discrepancy satisfy
\begin{gather}\label{eq:approxerror}
\bdistance[\xi]{x_\alpha}{x^\dagger} \leq \frac{\norm{w}^2}{2}\alpha,\\
\label{eq:approxerror_disrc} \norm{Kx_\alpha-y^\dagger}\leq
2\norm{w}\alpha.
\end{gather}
With the choice $\xi_\alpha = -K^*(Kx_\alpha-y^\dagger)/\alpha$, the
data error and the corresponding discrepancy satisfy
  \begin{gather}
    \label{eq:dataerror}
    \bdistance[\xi_\alpha]{x_\alpha^\delta}{x_\alpha}\leq\frac{\delta^2}{2\alpha},\\
    \label{eq:dataerror_disrc}
    \norm{K(x_\alpha^\delta-x_\alpha)}\leq 2\delta.
  \end{gather}
\end{proposition}
\begin{proof}
  Inequalities~\eqref{eq:approxerror} and~\eqref{eq:approxerror_disrc}
  have been shown in~\cite{burger2004convarreg}, however, we include
  a short proof for the sake of completeness. By the minimizing property of
  $x_\alpha$ and the fact $Kx^\dagger = y^\dagger$ we have
  \[
  \tfrac12\norm{Kx_\alpha - y^\dagger}^2 + \alpha R(x_\alpha) \leq \alpha R(x^\dagger).
  \]
  Rearranging the terms and noting $\xi\in\partial R(x^\dagger)$
  yields
  \begin{equation}
    \label{eq:aux_est_approx_error}
    \tfrac12\norm{Kx_\alpha - y^\dagger}^2 + \alpha\bdistance[\xi]{x_\alpha}{x^\dagger}
    \leq -\alpha \scp{\xi}{x_\alpha-x^\dagger}.
  \end{equation}
By observing the non-negativity of the Bregman distance
$\bdistance[\xi]{x_\alpha}{x^\dagger}$ and using $\xi=K^*w$ and
Cauchy-Schwarz inequality, we obtain
  \[
  \tfrac12\norm{Kx_\alpha-y^\dagger}^2 \leq \alpha\norm{w}\norm{Kx_\alpha-y^\dagger}
  \]
  which shows estimate~\eqref{eq:approxerror_disrc}.

Appealing again to inequality~\eqref{eq:aux_est_approx_error} and
using $\xi=K^*w$, Cauchy-Schwarz and Young's inequalities, we arrive
at
  \[
  \tfrac12\norm{Kx_\alpha - y^\dagger}^2 + \alpha\bdistance[\xi]{x_\alpha}{x^\dagger}
  \leq \tfrac{\alpha^2\norm{w}^2}{2} + \tfrac12\norm{Kx_\alpha -
  y^\dagger}^2.
  \]
  This establishes estimate~\eqref{eq:approxerror}.

  Next we use the minimizing property of $x_\alpha^\delta$ to get
  \[
  \tfrac12\norm{Kx_\alpha^\delta-y^\delta}^2 +
  \alpha\bdistance[\xi_\alpha]{x_\alpha^\delta}{x_\alpha} \leq
  \tfrac12\norm{Kx_\alpha-y^\delta}^2 -
  \alpha\scp{\xi_\alpha}{x_\alpha^\delta-x_\alpha}.
  \]
  From the optimality of $x_\alpha$, we have
  $\xi_\alpha=-K^*(Kx_\alpha - y^\dagger)/\alpha\in\partial
  R(x_\alpha)$. Plugging in $\xi_\alpha$ and rearranging the formula gives
  \begin{align*}
    \tfrac12\norm{Kx_\alpha^\delta-y^\delta}^2 +
    \alpha\bdistance[\xi_\alpha]{x_\alpha^\delta}{x_\alpha}
    & \leq \tfrac12\norm{Kx_\alpha-y^\delta}^2 + \scp{Kx_\alpha-y^\dagger}{K(x_\alpha^\delta-x_\alpha)}\\
    & = \tfrac12\norm{Kx_\alpha^\delta - y^\delta}^2 -
    \tfrac12\norm{K(x_\alpha^\delta - x_\alpha)}^2 \\
    & \qquad - \scp{y^\dagger-y^\delta}{K(x_\alpha^\delta - x_\alpha)},
  \end{align*}
  i.e.,
  \begin{equation}\label{eq:aux_est_data_error}
    \tfrac12\norm{K(x_\alpha^\delta - x_\alpha)}^2+
    \alpha\bdistance[\xi]{x_\alpha^\delta}{x_\alpha}
    \leq-\scp{y^\dagger-y^\delta}{K(x_\alpha^\delta - x_\alpha)}.
  \end{equation}
Now the non-negativity of the Bregman distance and Cauchy-Schwarz
inequality yields estimate~\eqref{eq:dataerror_disrc}. Next by
virtue of inequality~\eqref{eq:aux_est_data_error} and
Cauchy-Schwarz and Young's inequalities, we obtain
  \begin{equation*}
    \tfrac12\norm{K(x_\alpha^\delta - x_\alpha)}^2 + \alpha
    \bdistance[\xi]{x_\alpha^\delta}{x_\alpha} \leq
    \delta\norm{K(x_\alpha^\delta - x_\alpha)}\leq \tfrac{\delta^2}{2} +
    \tfrac12\norm{K(x_\alpha^\delta - x_\alpha)}^2
  \end{equation*}
  which concludes the proof.
\end{proof}

From~\cite{burger2004convarreg} we cite the following result.
\begin{proposition}[Estimate for the total error]
  \label{prop:estimate_total_error}
  If the source condition~\eqref{eq:source_condition1} holds with $\xi
  = K^* w \in \partial R(x^\dagger)$, then we have
  \begin{gather}
    \label{eq:totalerror}
    \bdistance[\xi]{x_\alpha^\delta}{x^\dagger} \leq
    \frac12\Bigl(\frac{\delta}{\sqrt{\alpha}} +
    \sqrt{\alpha}\norm{w}\Bigr)^2,\\
    \label{eq_totalerror_disrc}
    \norm{Kx_\alpha^\delta - y^\delta} \leq \delta + 2\alpha\norm{w}.
  \end{gather}
\end{proposition}

Although the Bregman distance does in general not fulfill the
triangle inequality we see that the total error
$\bdistance{x_\alpha^\delta}{x^\dagger}$ behaves like the sum of the
approximation error $\bdistance{x_\alpha}{x^\dagger}$ and the data
error $\bdistance{x_\alpha^\delta}{x_\alpha}$. Indeed there holds
for $a,b\geq0$ that $(a+b)^2/2 \leq a^2+b^2 \leq (a+b)^2$ and hence,
we see that the estimate~\eqref{eq:totalerror} behaves like the sum
of the estimates~\eqref{eq:approxerror} and~\eqref{eq:dataerror}.

The connection between the total error in the Bregman distance and
the approximation and data errors can be made a bit more precise. To
this end, we utilize the following lemma which is an immediate
consequence of the definition of the Bregman distance:
\begin{lemma}
  \label{lem:triangle_ineq_bregman}
  Let $\xi\in\partial R(x^\dagger)$ and $\zeta\in\partial R(x)$. Then
  there holds for any $x'$ that
  \[
  \bdistance[\xi]{x'}{x^\dagger} = \bdistance[\zeta]{x'}{x} +
  \bdistance[\xi]{x}{x^\dagger} + \scp{\xi-\zeta}{x-x'}.
  \]
\end{lemma}

The next result is a consequence of Lemma
\ref{lem:triangle_ineq_bregman}, and will be used frequently.
\begin{corollary}
  \label{cor:splitting_approx_data_error}
  Let the source condition~\eqref{eq:source_condition1} be
  fulfilled. Then with the obvious choices of the
  respective subgradients, there holds
  \[
  \Bigl| \bdistance{x_\alpha^\delta}{x^\dagger} -
  \bigl(\bdistance{x_\alpha^\delta}{x_\alpha} +
  \bdistance{x_\alpha}{x^\dagger}\bigr) \Bigr| \leq 6\norm{w}\delta.
  \]
\end{corollary}
\begin{proof}
  Taking $x=x_\alpha$, $x' = x_\alpha^\delta$, $\xi = K^* w$ and
  $\xi_\alpha = -K^*(Kx_\alpha-y^\dagger)/\alpha$ in
  Lemma~\ref{lem:triangle_ineq_bregman} gives
  \[
  \bdistance{x_\alpha^\delta}{x^\dagger} =
  \bdistance{x_\alpha^\delta}{x_\alpha} +
  \bdistance{x_\alpha}{x^\dagger} + \scp{w + (Kx_\alpha -
    y^\dagger)/\alpha}{K(x_\alpha - x_\alpha^\delta)},
  \]
  which together with inequalities~\eqref{eq:approxerror_disrc} and~\eqref{eq:dataerror_disrc}
  gives
  \begin{align*}
    \abs{\scp{w + (Kx_\alpha - y^\dagger)/\alpha}{K(x_\alpha -
        x_\alpha^\delta)}}& \leq (\norm{w} +
    \norm{Kx_\alpha-y^\dagger}/\alpha)\norm{K(x_\alpha^\delta-x_\alpha)}\\
    & \leq 6\norm{w}\delta.
  \end{align*}
  This concludes the proof.
\end{proof}
Hence, the total error differs from the sum of approximation and
data errors only by a term of magnitude $\delta$. In general, the
difference can be either positive or negative and both cases are
observed in numerical experiments.

We shall need the following result on the function
$\alpha\mapsto\|Kx_\alpha^\delta-y^\delta\|$.
\begin{lemma}
  \label{lem:continuity_discrepancy}
  The function $\alpha\mapsto \|Kx_\alpha^\delta-y^\delta\|$ is
  monotonically increasing and uniformly bounded. Moreover, if
  $\mathcal{J}_\alpha$ has a unique minimizer, then it is also
  continuous at $\alpha$.
\end{lemma}
\begin{proof}
  Let $\tilde{x}$ be an $R$-minimizing element in $X$. By the
  minimizing property of $x_\alpha^\delta$, we have
  \begin{equation*}
    \tfrac{1}{2}\|Kx_\alpha^\delta-y^\delta\|^2+\alpha
    R(x_\alpha^\delta) \leq \tfrac{1}{2}\|K\tilde{x}-y^\delta\|^2+\alpha
    R(\tilde{x}),
  \end{equation*}
  and thus
  $0\leq\|Kx_\alpha^\delta-y^\delta\|\leq\|K\tilde{x}-y^\delta\|<+\infty$,
  and is uniformly bounded. The proof of the remaining assertion can
  be found in \cite{Bonesky2009,jin2009c}.
\end{proof}

The last result in this section gives an estimate for the distance
between two regularized solutions for the same data but different
regularization parameters. This estimate underlies the quasi-optimality
principle in Section~\ref{sec:quasi-optim-princ}.
\begin{proposition}
  \label{prop:estimate_two_regularizations}
  For $q\in]0,1[$ and $\xi_\alpha^\delta = -K^*(Kx_\alpha^\delta-y^\delta)/\alpha$ there holds
  \begin{equation}
    \label{eq:two_reg_error}
    \bdistance[\xi_\alpha^\delta]{x_{q\alpha}^\delta}{x_\alpha^\delta} \leq
    \frac{(1-q)^2\norm{Kx_\alpha^\delta-y^\delta}^2}{2\alpha q}.
  \end{equation}
  Moreover, if the source condition \eqref{eq:source_condition1} is fulfilled, then
  \begin{equation}
    \label{eq:two_reg_error_disrc}
    \norm{K(x_{q\alpha}^\delta - x_\alpha^\delta)} \leq 2(1-q)(\delta +
    2\alpha\norm{w}).
  \end{equation}
\end{proposition}
\begin{proof}
  The minimizing property of $x_{q\alpha}^\delta$ implies
  \[
  \tfrac12\norm{Kx_{q\alpha}^\delta - y^\delta}^2 + q\alpha
  R(x_{q\alpha}^\delta) \leq \tfrac12\norm{Kx_{\alpha}^\delta -
    y^\delta}^2 + q\alpha R(x_{\alpha}^\delta).
  \]
  Rearranging the terms gives
  \[
  \tfrac12\norm{Kx_{q\alpha}^\delta - y^\delta}^2 +
  q\alpha\bdistance[\xi_\alpha^\delta]{x_{q\alpha}^\delta}{x_\alpha^\delta}
  \leq \tfrac12\norm{Kx_\alpha^\delta - y^\delta}^2 +
  q\scp{Kx_\alpha^\delta-y^\delta}{K(x_{q\alpha}^\delta -
    x_\alpha^\delta)},
  \]
  which leads to
  \[
  q\alpha\bdistance[\xi_\alpha^\delta]{x_{q\alpha}^\delta}{x_\alpha^\delta} \leq
  -(1-q)\scp{Kx_\alpha^\delta - y^\delta}{K(x_{q\alpha}^\delta -
    x_\alpha^\delta)} - \tfrac12\norm{K(x_{q\alpha}^\delta -
    x_\alpha^\delta)}^2.
  \]
  Appealing again to Cauchy-Schwarz and Young's inequalities
  gives~\eqref{eq:two_reg_error}. Using Cauchy-Schwarz inequality
  in
  \[
  \tfrac12\norm{K(x_{q\alpha}^\delta - x_\alpha^\delta)}^2 \leq
  -(1-q)\scp{Kx_\alpha^\delta - y^\delta}{K(x_{q\alpha}^\delta -
    x_\alpha^\delta)},
  \]
  and noting estimate \eqref{eq_totalerror_disrc} shows the
  remaining assertion.
\end{proof}

\section{A parameter choice \'a la Hanke-Raus}
\label{sec:parameter-choice-a} In this section, we investigate a
first heuristic parameter choice rule based on error estimate, which
resembles a rule due to Hanke and Raus \cite{Hanke1996}.
Although it is known that heuristic rules can never lead to
regularization methods in the context of the classical worst-case
scenario unless the problem is well-posed~\cite{Bakushinskii1984},
they have proven applicable and useful in practice
\cite{hansen:1998}. Recent results \cite{Kindermann2008} show that
weak assumptions on the true data $y^\dagger$ as well as the noisy
data $y^\delta$, hence leaving the worst-case scenario analysis,
lead to provable error estimates. We shall establish a posteriori
error estimates as well as convergence for the rule.
\subsection{Motivation}
\label{sec:motivation}

We see from Proposition~\ref{prop:estimate_total_error} that the
estimate for the total error differs from that for the squared
residual by a factor of $1/\alpha$:
\begin{align*}
  \frac{\norm{Kx_\alpha^\delta - y^\delta}^2}{\alpha} & \leq
  \frac{(\delta + 2\alpha\norm{w})^2}{\alpha} =
  \Bigl(\frac{\delta}{\sqrt{\alpha}} + 2\sqrt{\alpha}\norm{w}\Bigr)^2\\
  & \approx \frac12\Bigl(\frac{\delta}{\sqrt{\alpha}} +
  \sqrt{\alpha}\norm{w}\Bigr)^2\geq
  \bdistance{x_\alpha^\delta}{x^\dagger}.
\end{align*}
Since the value $\norm{Kx_\alpha^\delta - y^\delta}^2/\alpha$ can be
evaluated a posteriori without resorting to any knowledge of the
exact noise level $\delta$, we propose to use it as an estimate of
the total error and to choose an appropriate regularization
parameter $\alpha$ by minimizing the function
\begin{equation}
  \label{eq:def_phi}
  \phi(\alpha) = \frac{\norm{Kx_\alpha^\delta-y^\delta}^2}{\alpha}.
\end{equation}
This resembles the parameter choice due to Hanke and Raus
\cite{Hanke1996} for classical Tikhonov regularization as well as
several iterative regularization methods.

In view of Lemma~\ref{lem:continuity_discrepancy} we see that
$\lim_{\alpha\rightarrow+\infty}\phi(\alpha)=0$. Similarly, in case
of a unique minimizer to the functional $\mathcal{J}_\alpha$ for any
$\alpha>0$, the optimization problem of minimizing $\phi$
over any bounded and closed interval of the positive semi-axis
$\RR_+$ is well-defined.

\subsection{A posteriori error estimates}
\label{sec:error-estim-conv}

In this part, we derive a posteriori error estimates for the
Hanke-Raus rule to offer partial
theoretical justification. We shall treat two cases of uniformly
convex $R$ and the particular case $R(x)=\|x\|_{\ell^1}$ separately.

\begin{theorem}
  \label{thm:estimate_hanke_raus}
  Let the source condition~\eqref{eq:source_condition1} be fulfilled.
  Let $\phi$ be defined by~\eqref{eq:def_phi} and $\alpha^*$
  defined as
  \begin{equation}
    \label{eq:def_HR}
    \alpha^* \in \argmin_{\alpha\in[0,\norm{K}^2]} \phi(\alpha).
  \end{equation}
  If furthermore $\delta^* := \norm{Kx_{\alpha^*}^\delta -
    y^\delta}\neq 0$ then there exists a constant $C>0$ such that
  \[
  \bdistance{x_{\alpha^*}^\delta}{x^\dagger} \leq C\left(1 +
  \bigl(\tfrac{\delta}{\delta^*}\bigr)^2\right) \max(\delta,\delta^*).
  \]
\end{theorem}
\begin{proof}
  We have from Corollary~\ref{cor:splitting_approx_data_error}
  \[
  \bdistance{x_\alpha^\delta}{x^\dagger} \leq
  \bdistance{x_\alpha}{x^\dagger} +
  \bdistance{x_\alpha^\delta}{x_\alpha} + 6\norm{w}\delta.
  \]
  It suffices to estimate the two Bregman distance terms.
  First we estimate the approximation error
  $\bdistance{x_{\alpha^\ast}}{x^\dagger}$ for $\alpha=\alpha^*$. By
  inequalities~\eqref{eq:two_reg_error_disrc}
  and~\eqref{eq_totalerror_disrc}, we obtain
  \begin{align*}
    \bdistance{x_{\alpha^*}}{x^\dagger} & \leq
    \norm{w}\norm{Kx_{\alpha^*}-y^\dagger} \\
    &\leq \norm{w}\left(\norm{K(x_{\alpha^*}
      - x_{\alpha^*}^\delta)} + \norm{Kx_{\alpha^*}^\delta - y^\delta} + \delta\right)\\
    & \leq \norm{w}(2\delta + \delta^*+\delta) \leq
    4\norm{w}\max(\delta,\delta^*).
  \end{align*}
  Next we estimate the data error
  $\bdistance{x_{\alpha^\ast}^\delta}{x_{\alpha^\ast}}$. Using
  inequality~\eqref{eq:dataerror}, we get
  \begin{equation}
    \label{eq:aux_est_hanke_raus}
    \bdistance{x_{\alpha^*}^\delta}{x_{\alpha^*}} \leq \frac{\delta^2}{2\alpha^*}
    = \Bigl(\frac{\delta}{\delta^*}\Bigr)^2\, \frac{\norm{Kx_{\alpha^*}^\delta - y^\delta}^2}{2\alpha^*}.
  \end{equation}
  By the definition of $\alpha^*$, we only increase the right hand side
  if we replace $\alpha^*$ by any other $\bar \alpha\in
  [0,\norm{K}^2]$. We use $\bar\alpha = \bar c \delta$ with $\bar
  c=\min(1,\delta^{-1})\norm{K}^2$ and deduce
  from inequality~\eqref{eq_totalerror_disrc} that
  \[
  \norm{Kx_{\bar \alpha}^\delta-y^\delta}\leq (1+2\bar
  c\norm{w})\delta.
  \]
  Replacing $\alpha^*$ by $\bar \alpha$
  in inequality~\eqref{eq:aux_est_hanke_raus}, we have
  \[
  \bdistance{x_{\alpha^*}^\delta}{x_{\alpha^*}} \leq
  \Bigl(\frac{\delta}{\delta^*}\Bigr)^2\,\frac{\norm{Kx_{\bar\alpha}^\delta
      - y^\delta}^2}{2\bar\alpha} \leq
  \Bigl(\frac{\delta}{\delta^*}\Bigr)^2\, \frac{(1+2\bar
    c\norm{w})^2\delta}{2\bar c}.
  \]
  By combining the above two estimates, we finally arrive at
  \begin{align*}
    \bdistance{x_{\alpha^*}^\delta}{x^\dagger}& \leq
    4\norm{w}\max(\delta,\delta^*) +
    \Bigl(\frac{\delta}{\delta^*}\Bigr)^2\frac{(1+2\bar
      c\norm{w})^2}{2\bar c}
    \delta + 6\norm{w}\delta\\
  & \leq C(1 + \bigl(\tfrac{\delta}{\delta^*}\bigr)^2) \max(\delta,\delta^*)
  \end{align*}
  with $C = \max(10\norm{w}, (1+2\bar c\norm{w})^2/(2\bar c))$ as
  desired.
\end{proof}

The preceding result estimates the error in terms of the Bregman
distance. In the case of $p$-convex regularization terms $R$ (see
e.g., \cite{schuster2008tikbanach}), this also provides error
estimates in norm, i.e., $\norm{x_{\alpha^*}^\delta-x^\dagger}$.
However, the interesting case of $\ell^1$ regularization, i.e.,
$X=\ell^2$ and~$R(x) = \norm[\ell^1]{x} = \sum_{k}\abs{x_k}$
is not covered. In this
case the Bregman distance is not even positive definite,
i.e.,~$\bdistance{x'}{x}$ may vanish for distinct $x'$ and $x$.
However, by using techniques
from~\cite{lorenz2008reglp,grasmair2008sparseregularization}, we are
still able to prove an analogous error estimate for this case. To
this end, we recall the following result
\cite{grasmair2008sparseregularization}.

\begin{lemma}\label{lem:fbisparse}
Let $X=\ell^2$ and $R(x) = \sum_{k}\abs{x_k}$.
Assume that the solution $x^\dagger$ is finitely supported and satisfies the
source condition (\ref{eq:source_condition1}). Moreover, assume that the operator
$K$ satisfies the finite basis injectivity property, that is, for
any finitely support $u$ and $v$, there holds that $Ku=Kv$ implies
$u=v$. Then there exist two positive constants $c_1$ and $c_2$ such
that
\begin{equation*}
\norm[\ell^1]{x-x^\dagger}\leq c_1[R(x)-R(x^\dagger)]+c_2\norm{K(x-x^\dagger)}.
\end{equation*}
\end{lemma}

We are now ready to transfer Theorem~\ref{thm:estimate_hanke_raus} to the case
$R(x)=\norm[\ell^1]{x}$.
\begin{theorem}\label{thm:hankrausestimate2}
  Assume that the conditions in Lemma \ref{lem:fbisparse} are
  satisfied. Let $\alpha^*$ be chose according to~\eqref{eq:aux_est_hanke_raus}.
  If furthermore $\delta^* := \norm{Kx_{\alpha^*}^\delta -
    y^\delta}\neq 0$ then there exists a constant $C>0$ such that
  \[
  \norm[\ell^1]{x^\dagger-x_{\alpha^*}^\delta} \leq C (1 +
  (\tfrac{\delta}{\delta^*})^2){\max(\delta^*,\delta)}.
  \]
\end{theorem}
\begin{proof}
  By Lemma \ref{lem:fbisparse}, the definition of Bregman distance
  $\bdistance{x}{x^\dagger}$ and the source condition
  \eqref{eq:source_condition1}, we have
  \begin{equation*}
    \begin{aligned}
      \norm[\ell^1]{x-x^\dagger}&\leq c_1[R(x)-R(x^\dagger)]+c_2\|K(x-x^\dagger)\|\\
      &=c_1\bdistance{x}{x^\dagger}+c_1\langle\xi,x-x^\dagger\rangle+c_2\|K(x-x^\dagger)\|\\
      &= c_1\bdistance{x}{x^\dagger}+c_1\langle
      w,K(x-x^\dagger)\rangle+c_2\|K(x-x^\dagger)\|\\
      &\leq{}c_1\bdistance{x}{x^\dagger}+(c_1\|w\|+c_2)\|K(x-x^\dagger)\|
    \end{aligned}
  \end{equation*}
  by Cauchy-Schwarz inequality. Now by virtue of
  Corollary~\ref{cor:splitting_approx_data_error} we have
  \begin{equation*}
    \norm[\ell^1]{x_{\alpha^\ast}^\delta-x^\dagger}\leq
    c_1(\bdistance{x_{\alpha^\ast}^\delta}{x_{\alpha^\ast}} +
    \bdistance{x_{\alpha^\ast}}{x^\dagger} + 6\|w\|\delta)
    +(c_1\|w\|+c_2)\|K(x_{\alpha^\ast}^\delta-x^\dagger)\|.
  \end{equation*}
  Next we bound each term on the right hand side. First observe
  \begin{align*}
    \|K(x_{\alpha^\ast}^\delta-x^\dagger)\|&\leq
    \|Kx_{\alpha^*}^\delta-y^\delta\|+\|y^\delta-Kx^\dagger\|\\
    &\leq\delta^\ast+\delta\leq 2\max(\delta,\delta^\ast).
  \end{align*}
  Then, for the approximation error
  $\bdistance{x_{\alpha^\ast}}{x^\dagger}$, we obtain as before
  \begin{align*}
    \bdistance{x_{\alpha^*}}{x^\dagger} & \leq
    \norm{w}\norm{Kx_{\alpha^*}-y}\\
    & \leq \norm{w}(\norm{K(x_{\alpha^*}
      - x_{\alpha^*}^\delta)} + \norm{Kx_{\alpha^*}^\delta - y^\delta}
    + \delta)\\
    & \leq \norm{w}(2\delta + \delta^*+\delta) \leq
    4\norm{w}\max(\delta,\delta^*).
  \end{align*}
  Finally, for the data error
  $\bdistance{x_{\alpha^\ast}^\delta}{x_{\alpha^\ast}}$, we obtain
  from inequality~\eqref{eq:dataerror} and the definition of $\alpha^\ast$
  \begin{equation}
    \bdistance{x_{\alpha^*}^\delta}{x_{\alpha^*}} \leq \frac{\delta^2}{2\alpha^*}
    = \Bigl(\frac{\delta}{\delta^*}\Bigr)^2\, \frac{\norm{Kx_{\alpha^*}^\delta - y^\delta}^2}{2\alpha^*}.
  \end{equation}
  By the minimizing property of $\alpha^\ast$, replacing $\alpha^*$ by
  any other $\bar\alpha\in [0,\norm{K}^2]$ only increases the right
  hand side. Setting $\bar\alpha=c_3\delta\in[0,\|K\|^2]$, then
  $\|Kx_{\bar\alpha}^\delta-y^\delta\|^2\leq c_4\delta$
  (cf.~\cite{grasmair2008sparseregularization}), which consequently gives
  \begin{equation*}
    \bdistance{x_{\alpha^*}^\delta}{x_{\alpha^*}}
    \leq \Bigl(\frac{\delta}{\delta^*}\Bigr)^2\,
    \frac{c_4^2}{2c_3}\delta\leq
    \frac{c_4^2}{2c_3}\Bigl(\frac{\delta}{\delta^*}\Bigr)^2\max(\delta,\delta^\ast).
  \end{equation*}
  Combining these three estimates we arrive at the desired inequality with
  $C=\max(12c_1\|w\|+2c_2,\tfrac{c_1c_4^2}{2c_3})$.
\end{proof}

As long as the discrepancy $\delta^\ast$ is of order $\delta$,
Theorems \ref{thm:estimate_hanke_raus} and
\ref{thm:hankrausestimate2} imply that the approximation
$x_{\alpha^\ast}^\delta$ with $\alpha^\ast$ chosen by the rule
\eqref{eq:def_phi} converges to the exact solution $x^\dagger$ at
the same rate as a priori parameter choice rules under identical
source conditions \cite{grasmair2008sparseregularization}. On the
other hand, if $\delta^\ast$ does not decrease as quickly as
$\delta$, then the convergence would be suboptimal. More dangerous
is the case that $\delta^\ast$ decreases more quickly. Then the
prefactor $\delta/\delta^\ast$ blows up, and the approximation may
diverge. Therefore, the value of $\delta^\ast$ should always be
monitored as an a posteriori criterion: The computed approximation
should be discarded if $\delta^\ast$ is deemed too small.

\subsection{Convergence}
\label{sec:convergence} By stipulating additional conditions on the
data $y^\delta$ as in reference~\cite{Hanke1996}, however, we can
get rid of the prefactor $\frac{\delta}{\delta^\ast}$ in the
estimates and even obtain convergence of the method. To show this,
we denote by $Q$ the orthogonal projection onto the orthogonal
complement of the closure of $\range K$.
\begin{corollary}
  \label{cor:existence_alpha*}
  If for the noisy data $y^\delta$, there exists some $\varepsilon>0$ such that
  \begin{equation*}
    \|Q(y^\dagger-y^\delta)\|\geq \varepsilon \|y^\dagger-y^\delta\|,
  \end{equation*}
  then $\alpha^\ast$ according to~\eqref{eq:def_HR} is positive.
  Moreover, under the conditions of
  Theorem~\ref{thm:estimate_hanke_raus}, there holds
  \begin{equation*}
    \bdistance{x_{\alpha^\ast}^\delta}{x^\dagger}\leq
    C\left(1+\frac{1}{\varepsilon^2}\right)\max(\delta,\delta^\ast),
  \end{equation*}
  and under the conditions of Theorem \ref{thm:hankrausestimate2}, there holds
  \begin{equation*}
    \|x_{\alpha^\ast}^\delta-x^\dagger\|\leq
    C\left(1+\frac{1}{\varepsilon^2}\right)\max(\delta,\delta^\ast),
  \end{equation*}
\end{corollary}
\begin{proof}
  We observe
  \begin{equation}
    \label{eq:est_discrepancy_below}
    \|Kx_\alpha^\delta-y^\delta\|\geq
    \|Q(K(x_\alpha^\delta)-y^\delta)\|=\|Qy^\delta\|=\|Q(y^\delta-y^\dagger)\|\geq
    \varepsilon \|y^\delta-y^\dagger\|.
  \end{equation}
  This shows $\delta^*\geq \epsilon\delta$ and especially
  that $\phi(\alpha)\rightarrow+\infty$ as $\alpha\rightarrow0$.
  Consequently, there exists a positive $\alpha^\ast$ minimizing
  $\phi(\alpha)$ over $[0,\|K\|^2]$. The remaining assertion follows from
  the preceding estimate and the respective error estimate.
\end{proof}

The next theorem shows the convergence of the rule under the
condition that $\|Q(y^\dagger-y^\delta)\|\geq\varepsilon
\|y^\dagger-y^\delta\|$ holds uniformly for the data $y^\delta$ as
$\delta$ tends to zero.
\begin{theorem}
  \label{thm:convergence_hanke_raus}
  Assume that the functional $\mathcal{J}_\alpha$ is coercive and has a unique
  minimizer. Furthermore, in the situation of Theorem~\ref{thm:estimate_hanke_raus}
  let the assumption of Corollary~\ref{cor:existence_alpha*} be fulfilled
  uniformly, i.e.,~there exists an $\epsilon>0$ such that for every
  $\delta >0$, the following inequality holds
  \begin{equation}\label{ass:ydproj}
  \norm{Q(y^\dagger-y^\delta)} \geq \epsilon\norm{y^\dagger-y^\delta}.
  \end{equation}
  Then there holds
  \[
  \bdistance{x_{\alpha^\ast(y^\delta)}^\delta}{x^\dagger}\to 0\ \text{
    for }\ \delta\to 0.
  \]
\end{theorem}
\begin{proof}
  By the definition of $\alpha^\ast$, we observe that the sequence
  $(\alpha^\ast\equiv\alpha^\ast(y^\delta))_{\delta\geq 0}$ is  uniformly
  bounded and hence, there exists an accumulation point $\bar\alpha$.
  We distinguish the two cases $\bar \alpha= 0$ and $\bar\alpha>0$.

  We first consider the case $\bar\alpha=0$. By Corollary~\ref{cor:splitting_approx_data_error},
  we split the error
  \begin{equation}
    \label{eq:aux_split_bdistance}
    \bdistance{x_{\alpha^\ast}^\delta}{x^\dagger} \leq
    \bdistance{x_{\alpha^\ast}^\delta}{x_{\alpha^*}} +
    \bdistance{x_{\alpha^\ast}}{x^\dagger} + 6\norm{w}\delta,
  \end{equation}
  and estimate the data and approximation errors separately.

  For the data error $\bdistance{x_{\alpha^\ast}^\delta}{x_{\alpha^\ast}}$,
  we deduce from inequality~\eqref{eq:est_discrepancy_below}
  and assumption \eqref{ass:ydproj} that
  \[
  \bdistance{x_{\alpha^\ast}^\delta}{x_{\alpha^\ast}}\leq
  \frac{\delta^2}{2\alpha^\ast}\leq \frac{\norm{Kx_{\alpha^\ast}^\delta -
      y^\delta}^2}{2\epsilon^2\alpha^\ast} =
  \frac{\phi(\alpha^\ast)}{2\epsilon^2}.
  \]
  Therefore, it suffices to show that $\phi(\alpha^\ast)$
  goes to zero as $\delta\rightarrow0$. By Proposition \ref{prop:estimate_total_error}, there holds for
  every $\alpha\in [0,\norm{K}^2]$ that
  \[
  \phi(\alpha^\ast)\leq \phi(\alpha)\leq
  \Bigl(\frac{\delta}{\sqrt{\alpha}} + 2\norm{w}\sqrt{\alpha}\Bigr)^2.
  \]
  Hence, we may choose $\alpha(\delta)$ in the usual way such that
  $\alpha(\delta)\to 0$ and $\delta^2/\alpha(\delta)\to 0$ for
  $\delta\to 0$. This shows $\phi(\alpha^\ast)\to 0$ for
  $\delta\to 0$.

  For the approximation error $\bdistance{x_{\alpha^\ast}}{x^\dagger}$,
  we deduce from the fact that $\bar\alpha=0$ and estimate~\eqref{eq:approxerror} that
  \[
  \bdistance{x_{\alpha^\ast}}{x^\dagger} \leq
  \frac{\alpha^\ast\norm{w}^2}{2} \to \frac{\bar\alpha\norm{w}^2}{2} = 0\
  \text{ for }\ \delta\to 0.
  \]
  Hence, all three terms on the right hand side of
  inequality~\eqref{eq:aux_split_bdistance} tend to zero for $\delta\to 0$
  as desired.

  Next we consider the remaining case $\bar\alpha>0$.  we use $\alpha^\ast\leq\norm{K}^2$ to get
  \[
  \phi(\alpha^\ast) \geq \frac{\norm{Kx_{\alpha^\ast}^\delta -
      y^\delta}^2}{\norm{K}^2}\geq 0.
  \]
  Since $\phi(\alpha^\ast)$ goes to zero for $\delta\to 0$ we deduce that
  $\norm{Kx_{\alpha^\ast}^\delta - y^\delta}$ tends to zero as well.
  Next by the minimizing property of $x_{\alpha^\ast}^\delta$, we have
  \begin{equation*}
    \tfrac{1}{2}\|Kx_{\alpha^\ast}^\delta-y^\delta\|^2+\alpha^\ast R(x_{\alpha^\ast}^\delta)
    \leq \tfrac{1}{2}\|Kx^\dagger-y^\delta\|^2+\alpha^\ast
    R(x^\dagger).
  \end{equation*}
  Therefore, both sequences $(\|Kx_{\alpha^\ast}^\delta -y^\delta\|)_\delta$ and
  $(R(x_{\alpha^\ast}^\delta))_\delta$ are uniformly bounded by
  the assumption $\bar{\alpha}>0$. By the coercivity of the functional $\mathcal{J}_\alpha$, the
  sequence $(x_{\alpha^\ast}^\delta)_\delta$ is uniformly bounded, and thus
  there exists a subsequence, possibly after relabeling as
  $(x_{\alpha^\ast}^\delta)_\delta$, that converges weakly to some
  $\hat{x}$. By the weak lower semicontinuity of the norm and the functional $R$, we have
  \begin{equation}\label{eqn:weaklsc}
    \|K\hat{x}-y^\dagger\|\leq\liminf_{\delta\rightarrow0}\|Kx_{\alpha^\ast}^\delta-y^\delta\|=0,\quad
    R(\hat{x})\leq
    \liminf_{\delta\rightarrow0}R(x_{\alpha^\ast}^\delta).
  \end{equation}
  Consequently, for any $x$
  \begin{equation}\label{eqn:minimizing}
    \begin{aligned}
      \tfrac{1}{2}\|K\hat{x}-y^\dagger\|^2+\bar{\alpha}R(\hat{x})&\leq\liminf_{\delta\rightarrow0}
      \tfrac{1}{2}\|Kx_{\alpha^\ast}^\delta-y^\delta\|^2+
      \liminf_{\delta\rightarrow0}\alpha^\ast R(x_{\alpha^\ast}^\delta)\\
      &\leq\liminf_{\delta\rightarrow0}\left\{\tfrac{1}{2}\|Kx_{\alpha^\ast}^\delta-y^\delta\|^2+\alpha^\ast
      R(x_{\alpha^\ast}^\delta)\right\}\\
      &\leq\liminf_{\delta\rightarrow0}\left\{\tfrac{1}{2}\|Kx-y^\delta\|^2+\alpha^\ast
      R(x)\right\}\\
      &=\tfrac{1}{2}\|Kx-y^\dagger\|^2+\bar{\alpha}R(x).
  \end{aligned}
  \end{equation}
  Hence $\hat{x}$ is a minimizer of the functional $\mathcal{J}_{\bar{\alpha}}$,
  and by the uniqueness of the minimizer, $\hat{x}=x_{\bar\alpha}$.
  Since this holds for every subsequence, the whole sequence converges
  weakly. Moreover, by the weak lower semicontinuity, we have
  \begin{equation}\label{eqn:convR}
    R(x_{\alpha^\ast}^\delta)\rightarrow R(x_{\bar\alpha}).
  \end{equation}
  Next we show that $x_{\bar\alpha}$ is an $R$-minimizing
  solution to the equation $Kx=y^\dagger$. However, this follows directly
  from inequality \eqref{eqn:weaklsc} that $\|Kx_{\bar\alpha}-y^\dagger\|=0$,
  and from inequality \eqref{eqn:minimizing}
  \begin{equation*}
    R(x_{\bar{\alpha}})\leq R(x)\quad \forall x,
  \end{equation*}
  which in particular by choosing $x$ in the set of $R$-minimizing
  solutions shows the claim. Now we deduce that
  \begin{equation*}
  \lim_{\delta\rightarrow0}\bdistance{x_{\alpha^\ast}^\delta}{x^\dagger}=
  \lim_{\delta\rightarrow0}\left(
  R(x_{\alpha^\ast}^\delta)-R(x^\dagger)-\scp{\xi}{x_{\alpha^\ast(y^\delta)}^\delta-x^\dagger}\right)=0.
  \end{equation*}
  by observing identity \eqref{eqn:convR} and the weak convergence
  of the sequence $(x_{\alpha^\ast(y^\delta)}^\delta)_\delta$ to
  $x^\dagger$. This concludes the proof of the theorem.
\end{proof}

\begin{remark}
In Theorem \ref{thm:convergence_hanke_raus}, the uniqueness
assumption on the functional $\mathcal{J}_\alpha$ can be relaxed as
equation \eqref{eqn:convR} holds for each weakly convergent
subsequence. We have utilized the uniqueness of $R$-minimizing
solution, which may also be dropped by restating the result as: then
there exists some $R$-minimizing solution $x^\dagger$, such that
  \[
   \bdistance{x_{\alpha^\ast}^\delta}{x^\dagger}\to 0\ \text{
    for }\ \delta\to 0.
  \]
\end{remark}

In our context we are able to further weaken the assumption~\eqref{ass:ydproj} on the noise.
\begin{corollary}
  \label{cor:existence_alpha*_var}
  If there exists an $\epsilon\in ]0,1[$ such that for
  all $z\in \closure{K(\dom \partial R)}$ the following inequality holds
  \begin{equation*}
    \scp{y^\delta-y^\dagger}{z} \leq (1-\epsilon)\norm{y^\delta-y^\dagger}\norm{z}.
  \end{equation*}
  then the minimizer $\alpha^\ast$ to
  $\phi(\alpha)$ is positive. Moreover, under the conditions of Theorem
  \ref{thm:estimate_hanke_raus}, there holds
  \begin{equation*}
    \bdistance{x_{\alpha^\ast}^\delta}{x^\dagger}\leq
    C\left(1+\frac{1}{1-(1-\epsilon)^2}\right)\max(\delta,\delta^\ast),
  \end{equation*}
  and under the conditions of Theorem \ref{thm:hankrausestimate2}, there
  holds
  \begin{equation*}
    \|x_{\alpha^\ast}^\delta-x^\dagger\|\leq
    C\left(1+\frac{1}{1-(1-\epsilon)^2}\right)\max(\delta,\delta^\ast),
  \end{equation*}
\end{corollary}
\begin{proof}
  By observing the fact that both $x_\alpha^\delta$ and $x^\dagger$ are in
  $\dom\partial R$ and the assumption on the noise $y^\dagger-y^\delta$, we derive
  \begin{align*}
    \norm{Kx_\alpha^\delta-y^\delta}^2 & = \norm{K(x_\alpha^\delta-x^\dagger) - (y^\delta-y^\dagger)}^2\\
    & = \norm{K(x_\alpha^\delta-x^\dagger)}^2 - 2\scp{K(x_\alpha^\delta-x)}{y^\delta-y^\dagger} +  \norm{y^\delta-y^\dagger}^2\\
    & \geq \norm{K(x_\alpha^\delta-x^\dagger)}^2 - 2(1-\epsilon)\norm{K(x_\alpha^\delta-x^\dagger)}\norm{y^\delta-y^\dagger}
      +  \norm{y^\delta-y^\dagger}^2\\
    & = (\norm{K(x_\alpha^\delta-x^\dagger)} - (1-\epsilon)\norm{y^\delta-y^\dagger})^2 + (1-(1-\epsilon)^2)\norm{y^\delta-y^\dagger}^2\\
    & \geq (1-(1-\epsilon)^2)\delta^2.
  \end{align*}
  This in particular implies $(\delta^*)^2\geq (1-(1-\epsilon)^2)\delta^2$ and
  consequently that $\phi(\alpha)\rightarrow+\infty$ as $\alpha\rightarrow0$.
  Therefore, there exists a positive $\alpha^\ast$ minimizing
  $\phi(\alpha)$ over $[0,\|K\|^2]$. The remaining assertion follows
  similar to the proof of Theorem~\ref{thm:convergence_hanke_raus}.
\end{proof}

\begin{remark}[Comparing the assumptions on the noise]
  In Corollary~\ref{cor:existence_alpha*} or
  Theorem~\ref{thm:convergence_hanke_raus} we assumed
  \begin{equation*}
    \exists \epsilon\in]0,1[\,\,\forall\delta>0:\ \norm{Q(y^\dagger-y^\delta)}\geq\epsilon\norm{y^\dagger-y^\delta}
  \end{equation*}
  which is, with $P$ denoting the orthogonal projector onto
  $\closure{\range K}$, equivalent to
  \begin{equation}
    \label{eq:noise_condition_rgK}
    \exists \epsilon'>]0,1[\,\,\forall\delta>0:\ \norm{P(y^\dagger-y^\delta)}\leq\epsilon'\norm{y^\dagger-y^\delta}.
  \end{equation}
  In Corollary~\ref{cor:existence_alpha*_var} we assumed
  \begin{equation}
    \label{eq:noise_condition_rgKdomR}
    \exists \epsilon''\in]0,1[\forall\delta>0\forall z\in\closure{K(\dom \partial R)}:\ \scp{y^\delta-y^\dagger}{z}
    \leq \epsilon''\norm{y^\delta-y^\dagger}\norm{z}.
  \end{equation}
  In the case $\dom \partial R = X$ we conclude from this assumption
  that~\eqref{eq:noise_condition_rgK} holds with $\epsilon' =
  \sqrt{1-\epsilon''}$. Hence, in this
  case~\eqref{eq:noise_condition_rgKdomR}
  implies~\eqref{eq:noise_condition_rgK}. However, if $\dom \partial
  R$ is strictly contained in $X$
  condition~\eqref{eq:noise_condition_rgKdomR} may be considerably
  weaker.
\end{remark}

\section{The quasi-optimality principle}
\label{sec:quasi-optim-princ}

In this part, we derive another error-estimate based heuristic
choice rule, i.e., the quasi-optimality principle, and discuss its
convergence properties. The motivation of the principle is as
follows: By Proposition~\ref{prop:estimate_two_regularizations} for
any $q\in]0,1[$, there holds
\[
\bdistance{x_{q\alpha}^\delta}{x_\alpha^\delta}\leq
\frac{(1-q)^2}{2q}\phi(\alpha).
\]
In particular, for a geometrically decreasing sequence of
regularization parameters, the Bregman distances of two consecutive
regularized solutions are bounded from above by a constant times the
estimator $\phi$. This suggests itself a parameter choice rule which
resembles the classical quasi-optimality
criterion~\cite{Tikhonov1977illposed,Tikhonov1979}. More precisely,
for given data $y^\delta$ and $q\in]0,1[$ we define a
\emph{quasi-optimality sequence} as
\[
\mu_k = \bdistance{x_{q^k}^\delta}{x_{q^{k-1}}^\delta}.
\]
The quasi-optimality principle consists of choosing the
regularization parameter $\alpha^{\mathrm{qo}} = q^k$ such that
$\mu_k$ is minimal over a given range $k\geq k_0$.
\begin{remark}
  The classical quasi-optimality principle as e.g.,~stated
  in~\cite{Tikhonov1977illposed,Tikhonov1979}, chooses
  $\alpha^{\mathrm{qo}}$ such that the quantity
  $\norm{\alpha\frac{dx_\alpha^\delta}{d\alpha}}$ is minimal. In our
  setting this approach seems not applicable since the mapping
  $\alpha\mapsto x_\alpha^\delta$ is in general not differentiable.
  For instance, in the case of $\ell^1$ regularization, the solution
  path, i.e., $x_\alpha^\delta$ with respect to $\alpha$, is piecewise
  linear~\cite{efron2004lars}. Hence we resort to the discrete version
  which is also used in~\cite{Bauer2009}.
\end{remark}

We shall follow closely the lines of reference~\cite{Glasko1984} and
start with some basic observations of the quasi-optimality sequence.
The quasi-optimality sequence for the exact data will be denoted by
\[
\mu_k^\dagger = \bdistance{x_{q^k}}{x_{q^{k-1}}}.
\]

\begin{lemma}
  \label{lem:asymptotics_mu}
  Let the source condition \eqref{eq:source_condition1} be satisfied. Then
  the quasi-optimality sequences $(\mu_k)_k$ and $(\mu_k^\dagger)_k$ fulfill
  \begin{enumerate}
  \item $\displaystyle\lim_{k\to-\infty}\mu_k = 0$,
  \item  $\displaystyle\lim_{k\to-\infty}\mu_k^\dagger = 0$ and $\displaystyle\lim_{k\to\infty}\mu_k^\dagger = 0$.
  \end{enumerate}
\end{lemma}
\begin{proof}
  Appealing to estimate~\eqref{eq:two_reg_error_disrc}, we have
  \[
  \mu_k \leq
  \frac{(1-q)^2}{2q}\frac{\norm{Kx_{q^{k-1}}^\delta-y^\delta}^2}{q^{k-1}}.
  \]
  Since the sequence $\left(\|Kx_{q^{k-1}}^\delta-y^\delta\|\right)_k$
  stays bounded for $k\to-\infty$, see Lemma
  \ref{lem:continuity_discrepancy}, the first claim follows directly. Setting
  $\delta=0$ in the above argumentation shows the first statement of
  Claim 2. Now we use inequality~\eqref{eq:approxerror_disrc} and
  estimate
  \[
  \mu_k^\dagger \leq \frac{(1-q)^2}{2}\frac{\norm{Kx_{q^{k-1}}-y^\dagger}^2}{q^k}
  \leq 2(1-q)^2\norm{w}^2q^{k-2}.
  \]
  This shows the second statement of Claim 2.
\end{proof}

Now we show that the quasi-optimality sequences for exact and noisy
data approximate each other for vanishing noise level.
\begin{lemma}
  \label{lem:qo_mu_approx}
  Let the source condition \eqref{eq:source_condition1} be
  satisfied. Then for any $k_1\in\ZZ$, there holds
  \[
  \lim_{\delta\to 0}\sup_{k\leq k_1}\abs{\mu_k-\mu_k^\dagger} = 0.
  \]
\end{lemma}
\begin{proof}
We will use the abbreviations $x_k^\delta = x_{q^k}^\delta$, $x_k =
x_{q^k}$, $\xi_k^\delta = -K^*(Kx_k^\delta - y^\delta)$ and $\xi_k =
-K^*(Kx_k-y)$ to simplify the notation. By the definition of $\mu_k$
and $\mu_k^\dagger$, we have
  \begin{align*}
    \abs{\mu_k-\mu_k^\dagger} & = \abs{\bdistance{x_k^\delta}{x_{k-1}^\delta} - \bdistance{x_k}{x_{k-1}}}\\
    & = \bigl|R(x_k^\delta) - R(x_{k-1}^\delta) - \scp{\xi_{k-1}^\delta}{x_k^\delta - x_{k-1}^\delta}
     \\
     & \qquad - R(x_k) + R(x_{k-1}) + \scp{\xi_{k-1}}{x_k - x_{k-1}}\bigr|\\
    & = |\bdistance{x_k^\delta}{x_k} - \bdistance{x_{k-1}^\delta}{x_{k-1}} \\
    & \qquad + \scp{\xi_{k-1}-\xi_{k-1}^\delta}{x_k^\delta-x_{k-1}^\delta}
    + \scp{\xi_k-\xi_{k-1}}{x_k^\delta-x_k}|.
  \end{align*}
  Now we estimate all four terms separately. Using inequality~\eqref{eq:dataerror}
  we can bound the first two terms by
  \[
  \bdistance{x_k^\delta}{x_k} \leq \frac{\delta^2}{2q^k},\qquad
  \bdistance{x_{k-1}^\delta}{x_{k-1}} \leq \frac{\delta^2}{2q^{k-1}}.
  \]
  For the third term, we get from estimates~\eqref{eq:dataerror_disrc} and~\eqref{eq:two_reg_error_disrc}
  \begin{align*}
    \scp{\xi_{k-1}-\xi_{k-1}^\delta}{x_k^\delta-x_{k-1}^\delta} & =
    \scp{K(x_{k-1}-x_{k-1}^\delta) + (y^\delta-y^\dagger)}{K(x_k^\delta-x_{k-1}^\delta)}/q^{k-1}\\
    & \leq (\norm{K(x_{k-1}-x_{k-1}^\delta)} +
    \delta)\norm{K(x_k^\delta-x_{k-1}^\delta)}/q^{k-1}\\
    & \leq 6\delta(1-q)(\delta + 2q^k\norm{w})/q^{k-1}.
  \end{align*}
  Similarly, we can estimate the last term by
  \[
  \scp{\xi_k-\xi_{k-1}}{x_k^\delta-x_k} \leq 8\delta(1-q)\norm{w}/q.
  \]
  Hence, all four terms are bounded for $k\leq k_1$ and decrease to zero
  as $\delta\to 0$. This proves the claim.
\end{proof}

In general, the quasi-optimality sequences $(\mu_k)_k$ and
$(\mu^\dagger_k)_k$ can vanish for finite indices $k$. Fortunately,
their positivity can be guaranteed for a class of functionals $R$.




\begin{lemma}\label{lem:positivity}
  Let the functional $R$ be $p$-convex, $R(x)=0$ only for $x=0$ and
  satisfy that for any $x$ tha value $\scp{\xi}{x}$ is independent of
  the choice of $\xi\in\partial R(x)$.
  If the data $y^\dagger$ (resp.~$y^\delta$) admits nonzero $\alpha^\ast$ for
  which $x_{\alpha^\ast}\neq0$, then $\mu_k^\dagger>0$ (resp.~$\mu_k>0$)
  for all $k\geq [\ln\alpha^\ast/\ln q]$.
\end{lemma}
\begin{proof}
  By the optimality condition for $x_\alpha$, we have
  \begin{equation*}
    -K^\ast(Kx_\alpha-y^\dagger)\in \alpha\partial R(x_\alpha).
  \end{equation*}
  By assumption, the value $\scp{\xi_\alpha}{x_\alpha}$ is independent
  of the choice of $\xi_\alpha\in\partial R(x_\alpha)$ and hence,
  taking duality pairing with $x_\alpha$ gives for any
  $\xi_\alpha\in\partial R(x_\alpha)$
  \begin{equation*}
    \scp{Kx_\alpha}{Kx_\alpha-y^\dagger}+\alpha
    \scp{\xi_\alpha}{x_\alpha}=0.
  \end{equation*}
  For non-zero $x_\alpha$ we have that $\scp{\xi_\alpha}{x_\alpha}$ is
  non-zero and hence, we get
  \begin{equation}\label{eqn:alpha}
   \alpha=\frac{\scp{Kx_\alpha}{Kx_\alpha-y^\dagger}}{\scp{\xi_\alpha}{x_\alpha}}
  \end{equation}
  Next by the assumption that the
  data $y^\dagger$ admits nonzero $\alpha^\ast$ for which
  $x_{\alpha^\ast}\neq0$, then for any $\alpha<\alpha^\ast$, $0$
  cannot be a minimizer of the Tikhonov functional. To see this, we
  assume that $0$ is a minimizer, i.e.,
  \begin{align*}
    \frac{1}{2}\|y^\dagger\|^2&= \frac{1}{2}\|K0-y^\dagger\|^2+\alpha R(0)
    <\frac{1}{2}\|Kx_{\alpha^\ast}-y^\dagger\|^2+\alpha^\ast
    R(x_{\alpha^\ast}),
  \end{align*}
  by the strict positivity of $R$ for nonzero $x$.
  This contradicts the minimality of $x_{\alpha^\ast}$.
  Now let
  $\alpha_1,\alpha_2<\alpha^\ast$ be distinct. Then both sets
  $\{x_{\alpha_1}\}$ and $\{x_{\alpha_2}\}$ contain no zero element.
  Next we show that the two sets are disjointed. Assume that
  $\{x_{\alpha_1}\}$ and $\{x_{\alpha_2}\}$ intersects nontrivially,
  i.e., there exists some nonzero $\tilde{x}$ such that $\tilde{x}\in
  \{x_{\alpha_1}\}\cap\{x_{\alpha_2}\}$. Then by equation
  \eqref{eqn:alpha} and choosing any $\tilde\xi\in\partial
  R(\tilde{x})$, we have
  \begin{equation*}
    \alpha_1=\frac{\scp{K\tilde x}{K\tilde x-y^\dagger}}{\langle
      \tilde\xi,\tilde{x}\rangle}=\alpha_2,
  \end{equation*}
  which is in contradiction with the distinctness of $\alpha_1$ and
  $\alpha_2$. Therefore, for distinct $\alpha_1,\alpha_2<\alpha^\ast$,
  the sets $\{x_{\alpha_1}\}$ and $\{x_{\alpha_2}\}$ are disjointed.
  Consequently, we have
  \begin{equation*}
    \|x_{\alpha_1}-x_{\alpha_2}\|>0.
  \end{equation*}
  Now by the $p$-convexity of $R$, we deduce for $q^k\leq \alpha^\ast$ that
  \begin{equation*}
    \mu^\dagger_k=D(x_{q^{k}},x_{q^{k-1}})\geq
    C\|x_{q^{k}}-x_{q^{k-1}}\|^p>0,
  \end{equation*}
  which shows the assertion for $\mu_k^\dagger$. The claim for $\mu_k$
  can be shown similarly.
\end{proof}

\begin{remark}
  The assumptions on $R$ in Lemma \ref{lem:positivity} are satisfied
  for many commonly used regularization functionals, e.g., $\|x\|_{\ell^p}$,
  $\|x\|_{L^p}$ with $p>1$ and the elastic-net functional \cite{Jin2009e}.
  However, the special case of $\|x\|_{\ell^1}$ is
  not covered. Indeed, the $\ell^1$ minimization can retrieve the
  support of the exact solution for sufficiently small noise level
  $\delta$ and $\alpha$, see \cite{Trede2010}. Consequently, both $\mu_k$ and
  $\mu^\dagger_k$ vanish for sufficiently large $k$, due to the lack
  of $p$-convexity. The bound $\alpha^\ast$ depends on $y(y^\delta)$,
  and for nonvanishing $y(y^\delta)$ can be either positive or $+\infty$, see
  \cite{jin2009c} for some discussions. The choice of $k_0$ should
  be related to $\alpha^\ast$ such that $\mu_{k_0}$ ($\mu_{k_0}^\dagger$)
  is nonzero.
\end{remark}

By combining the above two lemmas, we have the following important
corollary, which will play a key role in establishing the
convergence result.
\begin{corollary}
  \label{cor:alphaqo_to_zero} Under the conditions of Lemma
  \ref{lem:positivity}, the parameter $\alpha^{\mathrm{qo}}$ chosen by
  the quasi-optimality principle satisfies that for any sequence
  $\delta_n\to 0$ there holds that $\alpha^{\mathrm{qo}}\to 0$.
\end{corollary}
\begin{proof}
  By definition it holds that $\alpha^{\mathrm{qo}} = q^{k^*}$ where
  $k^*$ is such that the sequence $\mu_k$ is minimal.

  Observe that $\mu_k\leq \mu^\dagger_k + \abs{\mu_k-\mu^\dagger_k}$.
  Now, let $\epsilon>0$. Due to Lemma~\ref{lem:asymptotics_mu} there
  holds that $\mu^\dagger_k\to 0$ for $k\to \infty$ and hence, there
  exists an integer $\underline k$ such that $\mu^\dagger_{\underline k}\leq
  \epsilon/2$. Moreover, due to Lemma~\ref{lem:qo_mu_approx}, for any $k_1$
  there is $\bar\delta>0$ such that $\abs{\mu_{k} - \mu^\dagger_{k}}\leq
  \epsilon/2$ for all $k\leq k_1$, in particular with $\underline{k}$.
  Hence $\mu_k \leq \mu^\dagger_k+|\mu_k^\dagger-\mu_k|<\varepsilon$
  for the same value of $\underline{k}$.

  By Lemma \ref{lem:positivity}, for any finite integer $k_1$, the set
  $\{\mu_k^\dagger\}_{k=k_0}^{k_1}$ is finite and positive, and thus
  there exists a constant $\sigma>0$ such that $\mu_k^\dagger>\sigma$ for
  $k=k_0,\dots,k_1$.  Lemma \ref{lem:qo_mu_approx} indicates that
  $\mu_k$ is larger than $\sigma/2$ for $k=k_0,\ldots,k_1$ and
  sufficiently small $\delta$. Thus the sequence $(\alpha^{\text{qo}})_{\delta_n}$ can
  contain terms on $\{q^k\}_{k=k_0}^{k_1}$ only if $\delta$ is not too
  small, since $\mu_k$ goes to zero as $\delta$ tends to zero. Since
  $k_1$ is chosen arbitrarily, this implies the desired assertion.
\end{proof}

As remarked earlier, it is in general impossible to show the
convergence of $x_\alpha^\delta \to x^\dagger$ for a heuristic
parameter choice in the context of worst-case scenario analysis. For
the quasi-optimality principle, Glasko et al~\cite{Glasko1984}
defined the notion of auto-regularizable set as a condition on the
exact as well as noisy data. In the case of the continuous
quasi-optimality principle this is the set of $y^\delta$ such that
\[
\frac{\norm{\alpha \frac{dx_\alpha^\delta}{d\alpha}-\alpha \frac{d
      x_\alpha}{d\alpha}}}{\norm{x_\alpha^\delta-x_\alpha}}\geq q >0
\]
holds uniformly in $\alpha$ and $\delta$. This abstract condition on
the exact data has been replaced by a condition on the noise
in~\cite{Bauer2009}.

In our setting, the following sets are helpful for proving convergence.
\begin{definition}
  For $r>0$, $q\in]0,1[$, $K:X\to Y$ and $y^\dagger\in\range K$ we define the sets
  \[
  \mathcal{D}_r = \set{y^\delta\in
    Y}{\forall k:\ \abs{\bdistance{x_{q^k}^\delta}{x_{q^{k-1}}^\delta} -
      \bdistance{x_{q^k}}{x_{q^{k-1}}}} \geq
    r\bdistance{x_{q^k}^\delta}{x_{q^k}}}.
  \]
\end{definition}
The condition $y^\delta\in \mathcal{D}_r$ can be regarded as a
discrete analogue of the above-mentioned auto-regularizable
condition. With the set $\mathcal{D}_r$ at hand, we can now show
another result on the asymptotic behavior of the quasi-optimality
sequence. The condition is that the noisy data belongs to some set
$\mathcal{D}_r$.
\begin{lemma}
  Let $y^\delta\in\mathcal{D}_r$ for some $r>0$ and assume that
  $R(x_\alpha^\delta)\to\infty\ \text{ for } \alpha\to 0.$ Then
  $\mu_k\to\infty$ for $k\to\infty$.
\end{lemma}
\begin{proof}
  We observe that
  \begin{equation}
    \label{eq:aux_mu_k_infty}
    r\bdistance{x_{q^k}^\delta}{x_{q^k}}\leq
    \abs{\bdistance{x_{q^k}^\delta}{x_{q^{k-1}}^\delta} -
      \bdistance{x_{q^k}}{x_{q^{k-1}}}} = \abs{\mu_k - \mu^\dagger_k}.
  \end{equation}
  By the definition of the Bregman distance,~\eqref{eq:approxerror_disrc}
  and~\eqref{eq:dataerror_disrc} we have for $\xi_\alpha =
  -K^*(Kx_\alpha - y^\dagger)/\alpha$ that
  \begin{align*}
    \bdistance[\xi_\alpha]{x_\alpha^\delta}{x_\alpha} & =
    R(x_\alpha^\delta)-R(x_\alpha) - \scp{\xi_\alpha}{x_\alpha^\delta-x_\alpha}\\
    & = R(x_\alpha^\delta)-R(x_\alpha) + \tfrac1\alpha\scp{Kx_\alpha-y^\dagger}{K(x_\alpha^\delta-x_\alpha)}\\
    & \geq R(x_\alpha^\delta)-R(x_\alpha) + 4\norm{w}\delta.
  \end{align*}
  Since $R(x_\alpha)$ is bounded for $\alpha\to 0$ we see that by
  assumption that
  $\bdistance[\xi_\alpha]{x_\alpha^\delta}{x_\alpha}\to\infty$ for
  $\alpha\to 0$. This means that for $k\to\infty$ there holds that
  $\bdistance{x_{q^k}^\delta}{x_{q^k}}\to\infty$ and since
  $\mu^\dagger_k\to 0$, the claim follows from~\eqref{eq:aux_mu_k_infty}.
\end{proof}

Now we are in  position to show the main result of this section,
i.e., convergence for the quasi-optimality principle.
\begin{theorem}
  Let $(\delta_n)_n$, $\delta_n>0$, be a sequence converging to zero such that
  $y^{\delta_n}\to y^\dagger\in \range K$ and
  $y^{\delta_n}\in\mathcal{D}_r$ for some $r>0$. Let
  $(\alpha_n^{\mathrm{qo}}=\alpha_n^{\mathrm{qo}}(y^{\delta_n}))_n$ be the
  sequence of regularization
  parameters chosen by the quasi-optimality principle.
  Then
  \[
  \lim_{n\to\infty}\bdistance{x_{\alpha_n^{\mathrm{qo}}}^{\delta_n}}{x^\dagger}=0.
  \]
\end{theorem}
\begin{proof}
  Denote $\alpha_n^{\mathrm{qo}}$ by $q^{k_n}$. Then by using
  Corollary~\ref{cor:splitting_approx_data_error}, we derive
  \begin{align*}
    \bdistance{x_{q^{k_n}}^{\delta_n}}{x^\dagger} & \leq
    \bdistance{x_{q^{k_n}}^{\delta_n}}{x_{q^{k_n}}} +
    \bdistance{x_{q^{k_n}}}{x^\dagger} + 6\norm{w}\delta_n\\
    & \leq
    \frac{1}{r}\abs{\bdistance{x_{q^{k_n}}^\delta}{x_{q^{k_n-1}}^\delta}
      -
      \bdistance{x_{q^{k_n}}}{x_{q^{k_n-1}}}} + \bdistance{x_{q^{k_n}}}{x^\dagger} + 6\norm{w}\delta_n\\
    & = \frac{1}{r}\abs{\mu_{k_n} - \mu^\dagger_{k_n}} +
    \bdistance{x_{q^{k_n}}}{x^\dagger} + 6\norm{w}\delta_n.
  \end{align*}
  Now all three terms on the right hand side tend to zero for $n\to\infty$
  (the first due to Lemma~\ref{lem:qo_mu_approx} and the second due to
  $q^{k_n} = \alpha_n^{\mathrm{qo}}\to0$ by
  Corollary~\ref{cor:alphaqo_to_zero}).
\end{proof}
This theorem shows that it is possible that the quasi-optimality
principle leads to convergence in the setting of convex variational
regularization. However, the important question on how the sets
$\mathcal{D}_r$ look like, and especially, under what circumstance
they are non-empty, remains open. In~\cite{Glasko1984,Bauer2009} the
authors use spectral theory to investigate this issue -- a tool
which is unfortunately unavailable in our general setting.

\section{Numerical experiments}
\label{sec:numer-exper}

We conducted several experiments to illustrate our theoretical
findings.

\subsection{Experiment 1: Accuracy of the estimates}
\label{sec:experiment-1}

In the first experiment we show sharpness of the estimates of the
approximation, data and total errors. Especially we illustrate how
the function $\phi$ from the Hanke-Raus rule approximates the total
error.

The setting is as follows: We consider a deconvolution problem with
sparsity constraints. In particular, the space $X$ is a sequence
space $\ell^2$ and $Y$ is the Hilbert space $L^2[0,1]$. The operator
under consideration is $K=AB:\ell^2\to L^2[0,1]$ where
$A:L^2[0,1]\to L^2[0,1]$ is a circular convolution operator which
convolves with a characteristic function of an interval of width
$0.2$ and $B:\ell^2\to L^2[0,1]$ is a Haar wavelet synthesis
operator. Hence, the operator $K$ takes a square summable sequence
$x$, uses it as the expansion coefficients with respect to an
orthonormal Haar wavelet basis and afterwards performs a circular
convolution. The regularization function $R$ is the
$\|\cdot\|_{\ell^p}$ norm, i.e.,
\[
R(x) = \sum_k \abs{x_k}^p
\]
which has, for $p>1$, a single valued subgradient $\partial R(x) =
\sett{p\sign(x)\abs{x}^{p-1}}$. In particular we have chosen $p=1.2$
to promote sparsity of the minimizers
(cf.~\cite{daubechies2003iteratethresh}) and to get a $p$-convex
functional simultaneously. To construct a solution $x^\dagger$
fulfilling the source condition~\eqref{eq:source_condition1}, we
started with a function $w\in L^2[0,1]$ and set $\xi = K^*w$. Then
$x^\dagger$ was defined as
\[
x^\dagger_k = \sign(\xi_k)\abs{\xi_k/p}^{1/(p-1)}.
\]
We discretized the problem to 512 wavelet coefficients.
Figure~\ref{fig:experiment1_data} shows the chosen $w\in L^2[0,1]$,
the function $Bx^\dagger\in L^2[0,1]$ and the exact data $y^\dagger
= ABx^\dagger\in L^2[0,1]$. Both vectors $\xi$ and $x^\dagger$
consist of 165 non-zero coefficients, however, their plots are
noninformative.

\begin{figure}
  \centering
  \begin{tikzpicture}[xscale=2.6,yscale=2]
    \draw[->] (-.1,0) -- (1.1,0);
    \draw[->] (0,-.6) -- (0,1.1) node[anchor=west] {$w$};
    \foreach \x in {-0.5,1}
    \draw (-.02,\x) -- (.02,\x) node[anchor=east] {\footnotesize${\x}$};
    \draw (1,.05) -- (1,-.05) node[anchor=north] {\footnotesize$1$};
    \draw plot  file {data/experiment1/w.dat};
  \end{tikzpicture}
  \begin{tikzpicture}[xscale=2.6,yscale=.2]
    \draw[->] (-.1,0) -- (1.1,0);
    \draw[->] (0,-3) -- (0,10) node[anchor=west] {$Bx^\dagger$};
    \foreach \x in {-2,2,4,6,8}
    \draw (-.02,\x) -- (.02,\x) node[anchor=east] {\footnotesize${\x}$};
    \draw (1,.05) -- (1,-.05) node[anchor=north] {\footnotesize$1$};
    \draw plot  file {data/experiment1/xdagger.dat};
  \end{tikzpicture}
  \begin{tikzpicture}[xscale=2.6,yscale=.2]
    \draw[->] (-.1,0) -- (1.1,0);
    \draw[->] (0,-3) -- (0,10) node[anchor=west] {$y^\dagger$};
    \foreach \x in {-2,2,4,6,8}
    \draw (-.02,\x) -- (.02,\x) node[anchor=east] {\footnotesize${\x}$};
    \draw (1,.05) -- (1,-.05) node[anchor=north] {\footnotesize$1$};
    \draw plot  file {data/experiment1/ydagger.dat};
  \end{tikzpicture}

  \caption{Experiment 1: Left: $w$ from the source condition. Middle:
    $Bx^\dagger$. Right: $y^\dagger$.}
  \label{fig:experiment1_data}
\end{figure}
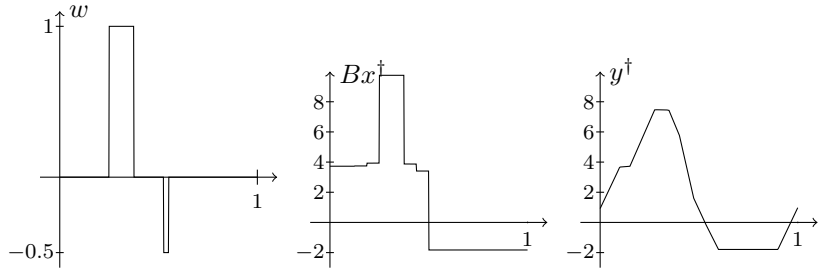

For a fixed noise level $\delta=0.02$, we generated noisy data
$y^\delta$ such that $\norm{y^\dagger-y^\delta}=\delta$. Then we
calculated minimizers $x_\alpha^\delta$ and $x_\alpha$ of the
Tikhonov functional with data $y^\delta$ and $y^\dagger$,
respectively, for different values of $\alpha$ with the combined
iterative hard- and soft-thresholding
from~\cite{lorenz2008conv_speed_sparsity}
(see~\cite{bredies2008harditer} for the iterative hard-thresholding
algorithm
and~\cite{daubechies2003iteratethresh,bredies2008itersoftconvlinear}
for the iterative soft-thresholding algorithm, the code is available
at
\url{http://www-public.tu-bs.de:8080/~dirloren/progs/iter_thresh.m}).
We calculated the different errors and the function $\phi$ from the
Hanke-Raus rule and show them in
Figure~\ref{fig:experiment1_results}. We observe that the function
$\phi$ captures the behavior of the total error very well. Moreover,
the sum of the approximation and data errors is close to the total
error. Surprisingly, the estimate from
Proposition~\ref{prop:estimate_total_error} is even closer to the
function $\phi$ than the total error itself---a result which is not
backed up by theory by now.

\begin{figure}
  \centering
  \begin{tikzpicture}[xscale=1,yscale=1]
    \draw[->] (-4.5,-6) -- (0.5,-6) node[below] {$\alpha$};
    \foreach \x in {-4,...,0}
    \draw (\x,-5.9) -- (\x,-6.1) node[below] {\footnotesize$10^{\x}$};
    \draw[->] (-4.5,-6) -- (-4.5,2);
    \foreach \y in {-5,...,1}
    \draw (-4.4,\y) -- (-4.6,\y) node[left] {\footnotesize$10^{\y}$};
    \draw[red] plot file {data/experiment1/bdist_recerror.dat};
    \draw[red] (0,-0.2) node[right] {$\bdistance{x_\alpha^\delta}{x^\dagger}$};
    \draw[blue] plot file {data/experiment1/bdist_dataerror.dat};
    \draw[blue] (-1,-5.5) node[right] {$\bdistance{x_\alpha^\delta}{x_\alpha}$};
    \draw[orange] plot file {data/experiment1/bdist_apperror.dat};
    \draw[orange] (-3.5,-4.5) node[below] {$\bdistance{x_\alpha}{x^\dagger}$};
    \draw[brown] plot file {data/experiment1/HR_estimator.dat};
    \draw[brown] (0,2) node[right] {$\phi(\alpha)$};
    \draw[green] plot file {data/experiment1/bdist_recerror_estimator.dat};
    \draw[green] (0,1.4) node[right] {$(\delta/\sqrt(\alpha) + \sqrt(\alpha)\norm{w})^2/2$};
  \end{tikzpicture}
  \caption{Experiment 1: Illustration of the different errors in log-log scale.}
  \label{fig:experiment1_results}
\end{figure}
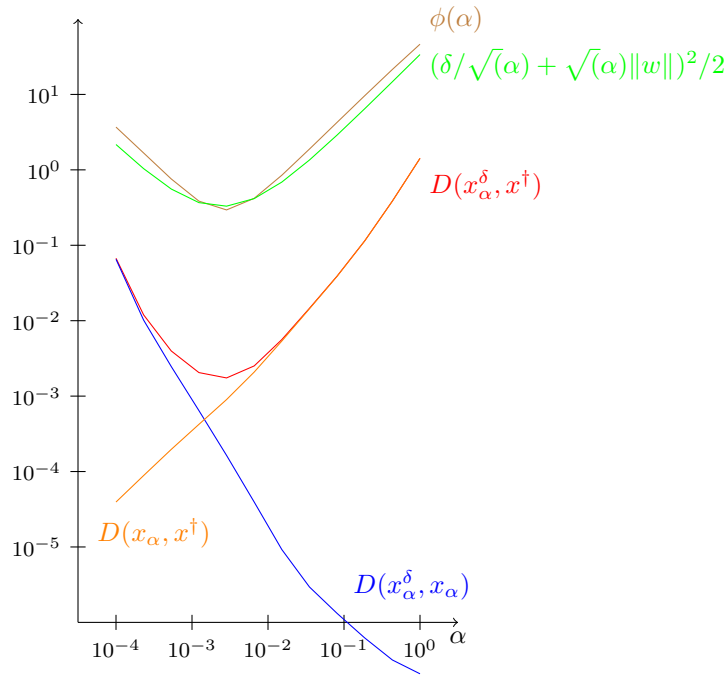

\begin{remark}
  The obtained results have been observed to be robust with respect to
  different noise realizations and different $w$ (if the obtained
  sparsity of the corresponding $x^\dagger$ is comparable).
\end{remark}

\subsection{Experiment 2: The Hanke-Raus rule}
\label{sec:experiment-2}

In this experiment we illustrate the performance of the Hanke-Raus
rule. We use the same set up as in the first experiment, i.e.,~the
same $x^\dagger$ and $K$. For a range of $\delta$ we generated noisy
data $y^\delta$ and calculated the regularization parameter
$\alpha^{\mathrm{HR}}$ with the Hanke-Raus rule of
Section~\ref{sec:parameter-choice-a} in a brute-force manner: we
tested values for $\alpha$ on a logarithmically uniform grid. As the
exact solution $x^\dagger$ is known in this case, we also calculated
the optimal regularization parameter $\alpha^{\mathrm{opt}}$,
i.e.,~the parameter $\alpha$ for which the error
$\bdistance{x_\alpha^\delta}{x^\dagger}$ is smallest, see
Figure~\ref{fig:experiment2_results} for the results. It is observed
that the Hanke-Raus parameter follows the optimal parameter closely
in this example and accordingly the error of the Hanke-Raus rule is
close to the optimal error.


\begin{figure}
  \centering
  \begin{tikzpicture}[xscale=.8,yscale=.8]
    \draw[->] (-5,-6) -- (-0.2,-6) node[below] {$\delta$};
    \foreach \x in {-4,...,-1}
    \draw (\x,-5.9) -- (\x,-6.1) node[below] {\footnotesize$10^{\x}$};
    \draw[->] (-5,-6) -- (-5,-.5);
    \foreach \y in {-5,...,-1}
    \draw (-4.9,\y) -- (-5.1,\y) node[left] {\footnotesize$10^{\y}$};
    \draw[red] plot file {data/experiment2/alpha_HR.dat};
    \draw[red] (-1,-1.6) node[right] {$\alpha^{\mathrm{HR}}$};
    \draw[blue] plot file {data/experiment2/alpha_opt.dat};
    \draw[blue] (-1,-2.2) node[right] {$\alpha^{\mathrm{opt}}$};
  \end{tikzpicture}~%
  \begin{tikzpicture}[xscale=0.8,yscale=0.8]
    \draw[->] (-5,-7) -- (-0.2,-7) node[below] {$\delta$};
    \foreach \x in {-4,...,-1}
    \draw (\x,-6.9) -- (\x,-7.1) node[below] {\footnotesize$10^{\x}$};
    \draw[->] (-5,-7) -- (-5,-1.5);
    \foreach \y in {-6,...,-2}
    \draw (-4.9,\y) -- (-5.1,\y) node[left] {\footnotesize$10^{\y}$};
    \draw[red] plot file {data/experiment2/error_HR.dat};
    \draw[red] (-1,-1.8) node[right] {$\bdistance{x_{\alpha^{\mathrm{HR}}}^\delta}{x^\dagger}$};
    \draw[blue] plot file {data/experiment2/error_opt.dat};
    \draw[blue] (-1,-2.4) node[right] {$\bdistance{x_{\alpha^{\mathrm{opt}}}^\delta}{x^\dagger}$};
  \end{tikzpicture}
  \caption{Experiment 2: Left: The regularization parameter by the Hanke-Raus rule
  and the optimal parameter in dependence of $\delta$. Right: The corresponding errors.}
  \label{fig:experiment2_results}
\end{figure}
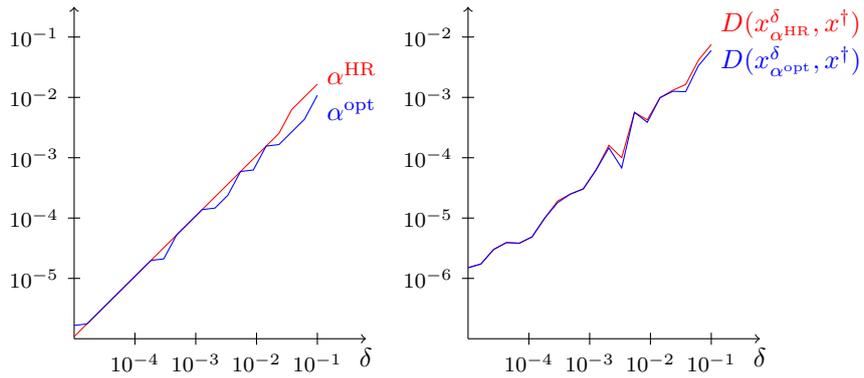

\subsection{Experiment 3: The quasi-optimality principle}
\label{sec:experiment-3}

This time the operator $K$, the data $x^\dagger$ and the
regularization function $R$ is again similar to Experiments~1 and 2.
Here we analyze how the quasi-optimality principle from
Section~\ref{sec:quasi-optim-princ} performs in practice. We chose
$\alpha_0 = 100\cdot\delta$ and $q=0.8$. Then we calculated
minimizers $x_{q^k\alpha_0}^\delta$ for several values of $k$ and
chose $\alpha^{\mathrm{qo}} = q^k\alpha_0$ as the one which
minimized
$\bdistance{x_{q^k\alpha_0}^\delta}{x_{q^{k-1}\alpha_0}^\delta}$.
Again, we also calculated the optimal value $\alpha^{\mathrm{opt}}$
of the regularization parameter and the corresponding errors, see
Figure~\ref{fig:experiment3_results} for the results. Again we
observed that this choice follows the optimal regularization
parameter closely and can produce accurate solutions.

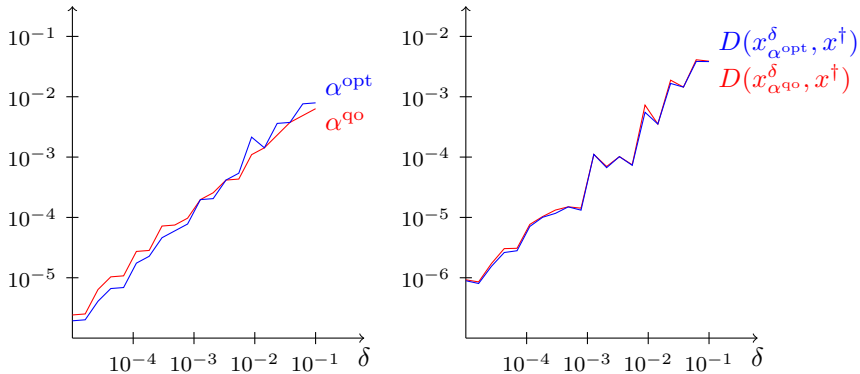
\begin{figure}
  \centering
  \begin{tikzpicture}[xscale=.8,yscale=.8]
    \draw[->] (-5,-6) -- (-0.2,-6) node[below] {$\delta$};
    \foreach \x in {-4,...,-1}
    \draw (\x,-5.9) -- (\x,-6.1) node[below] {\footnotesize$10^{\x}$};
    \draw[->] (-5,-6) -- (-5,-.5);
    \foreach \y in {-5,...,-1}
    \draw (-4.9,\y) -- (-5.1,\y) node[left] {\footnotesize$10^{\y}$};
    \draw[red] plot file {data/experiment3/alpha_qo.dat};
    \draw[red] (-1,-2.4) node[right] {$\alpha^{\mathrm{qo}}$};
    \draw[blue] plot file {data/experiment3/alpha_opt.dat};
    \draw[blue] (-1,-1.8) node[right] {$\alpha^{\mathrm{opt}}$};
  \end{tikzpicture}~%
  \begin{tikzpicture}[xscale=0.8,yscale=0.8]
    \draw[->] (-5,-7) -- (-0.2,-7) node[below] {$\delta$};
    \foreach \x in {-4,...,-1}
    \draw (\x,-6.9) -- (\x,-7.1) node[below] {\footnotesize$10^{\x}$};
    \draw[->] (-5,-7) -- (-5,-1.5);
    \foreach \y in {-6,...,-2}
    \draw (-4.9,\y) -- (-5.1,\y) node[left] {\footnotesize$10^{\y}$};
    \draw[red] plot file {data/experiment3/error_qo.dat};
    \draw[red] (-1,-2.7) node[right] {$\bdistance{x_{\alpha^{\mathrm{qo}}}^\delta}{x^\dagger}$};
    \draw[blue] plot file {data/experiment3/error_opt.dat};
    \draw[blue] (-1,-2.1) node[right] {$\bdistance{x_{\alpha^{\mathrm{opt}}}^\delta}{x^\dagger}$};
  \end{tikzpicture}
  \caption{Experiment 3: Left: The regularization parameter by
    the quasi-optimality criterion and the optimal parameter in
    dependence of $\delta$. Right: The corresponding errors.}
  \label{fig:experiment3_results}
\end{figure}

\subsection{Experiment 4: Deblurring with elastic net}
\label{sec:experiment-4}


In this experiment we used a standard problem from the Regularization
Tools toolbox by P.C. Hansen~\cite{Hansen2007}, namely the
\texttt{blur} problem. We used the parameters \texttt{N=50},
\texttt{band=5}, \texttt{sigma=1.2} and employed the so-called
elastic-net regularization~\cite{Zou.etal:2008,Jin2009e}, that is a
penalty term
\[
R(x) = \norm[1]{x} + \frac{\eta}{2}\norm[2]{x}^2.
\]
On the one hand, this weighted sum of the one- and the two-norm can
be seen as a stabilization for one-norm regularization and on the
other hand, it leads to a kind of grouping effect, see
also~\cite{Zou.etal:2008,Jin2009e}.

We generated a noisy image $y^\delta$ (with $\delta=0.1$) and fixed
$\eta=10^{-3}$.  We used a regularized semismooth Newton method
(proposed in~\cite{griesse2008ssnsparsity} for the case $\eta=0$ and
generalized to $\eta>0$ in~\cite{Jin2009e}). Then we calculated
solutions for a range of $\alpha$ and determined the regularization
parameters according to the Hanke-Raus rule and the quasi-optimality
criterion.  Moreover, we calculated the parameter according to the
discrepancy principle~\cite{Morozov1966} (to compare with a
non-heuristic a-posteriori rule) and the optimal regularization
parameter with respect to the norm and the Bregman distance. We
report the results in Table~\ref{tab:blur_results} and
Figure~\ref{fig:blur_results}.

\begin{table}
  \centering
  \caption{Results for experiment 4.}
  \input{data/table_compare}
  \label{tab:blur_results}
\end{table}

\begin{figure}
  \centering
  \begin{tabular}{ccc}
    \includegraphics[width=4cm]{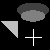} &
    \includegraphics[width=4cm]{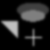} &
    \includegraphics[width=4cm]{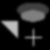} \\
    $x^\dagger$ &
    $y^\dagger$ &
    $y^\delta$ \\
    \includegraphics[width=4cm]{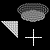} &
    \includegraphics[width=4cm]{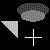} &
    \\
    smallest Bregman distance &
    smallest norm &
    \\
    \includegraphics[width=4cm]{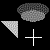} &
    \includegraphics[width=4cm]{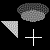} &
    \includegraphics[width=4cm]{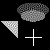} \\
    HR-rule &
    quasi-optimality &
    discrepancy principle
  \end{tabular}
  \caption{Results for the \texttt{blur} problem for the Hanke-Raus
    rule and the quasi-optimality criterion.}
  \label{fig:blur_results}
\end{figure}

We observe that all rules produce reasonable results and perform
comparably in terms of visual inspection. However, the numbers say a
little bit more: The discrepancy principle chooses a parameter which
is a bit too small and leads to larger errors both in terms of the
Bregman distance and the norm. The Hanke-Raus rule and the
quasi-optimality principle choose comparable parameters while the
quasi-optimality principle performs slightly better. Moreover, the
errors by the two proposed rules agree excellently with the optimal
one both in terms the Bregman distance and norm.

\section{Conclusion}
\label{sec:conclusion}

We have derived two error estimate-based heuristic parameter choice
rules for general convex variational regularization on the basis of
a refined analysis of the regularization process. These rules
reproduce the Hanke-Raus rule and the quasi-optimality criterion for
the conventional quadratic regularization. A posteriori error
estimates have been derived for the Hanke-Raus rule using the
Bregman distance. The convergence of both rules are discussed by
imposing conditions on the noisy data. Numerical results have
verified some theoretical findings and showed the effectiveness of
these rules. An important future research problem is to develop
efficient algorithms to numerically realize these rules. This is
nontrivial because the functionals under consideration are often
nonsmooth and there exists only an implicit relation between the
solution $x_\alpha^\delta$ and the regularization parameter
$\alpha$.

\bibliographystyle{plain}
\bibliography{literatur}

\end{document}

%% file: data/table_compare.tex
\begin{tabular}{@{}lccc@{}}
 \toprule
  & $\alpha$ & $\bdistance{x_\alpha^\delta}{x^\dagger}$ & $\norm[2]{x_\alpha^\delta-x^\dagger}$ \\
\midrule
 smallest Bregman distance & \verb|1.10e-02| & \verb|5.54e-02| & \verb|1.02e+01| \\
 smallest norm & \verb|3.20e-03| & \verb|7.39e-02| & \verb|7.36e+00| \\
 Hanke-Raus & \verb|2.61e-03| & \verb|9.03e-02| & \verb|7.47e+00| \\
 quasi-optimality & \verb|3.02e-03| & \verb|7.51e-02| & \verb|7.38e+00| \\
 discrepancy & \verb|9.29e-04| & \verb|7.16e-01| & \verb|1.01e+01| \\
\bottomrule
\end{tabular}